\documentclass[12pt]{amsart}
\usepackage[T1]{fontenc}
\usepackage{a4,amsmath}
\usepackage{graphicx}
\usepackage{tikz}
\usepackage{amsfonts,url} 
\usepackage{lscape}

\theoremstyle{plain}
\newtheorem{thm}{Theorem}
\newtheorem{lma}[thm]{Lemma}
\newtheorem{prop}[thm]{Proposition}
\newtheorem{kor}[thm]{Corollary}
\newtheorem{df}[thm]{Definition}
\newtheorem{constr}[thm]{Construction}
\newtheorem{conj}[thm]{Conjecture}
\newtheorem{quest}[thm]{Question}
\numberwithin{thm}{section}

\theoremstyle{remark}

\newtheorem{anm}[thm]{Remark}

\DeclareMathOperator{\dev}{dev}
\DeclareMathOperator{\supp}{supp}

\title{Triple arrays from difference sets}
\author{Tomas Nilson and Peter J. Cameron}
\begin{document}

\begin{abstract}

This paper addresses the question whether triple arrays can be constructed from Youden squares developed from difference sets. We prove that if the difference set is abelian, then having $-1$ as multiplier is both a necessary and sufficient condition for the construction to work. Using this, we are able to give a new infinite family of triple arrays. We also give an alternative and more direct version of the construction, leaving out the intermediate step via Youden squares. This is used when we analyse the case of non-abelian difference sets, for which we prove a sufficient condition for giving triple arrays. We do a computer search for such non-abelian difference sets, but have not found any examples satisfying the given condition. 

\end{abstract}
\maketitle

\begin{section}{Introduction}

Agrawal~\cite{Agrawal1} considered a type of row-column design which now goes by the name of ``triple array''. He observed that such an array implies the existence of a symmetric balanced incomplete block design, and asked whether triple arrays can be constructed from such designs. The statement that this can always be done, with exception for trivial and the smallest non-trivial design, is known as \emph{Agrawal's conjecture}.

\begin{df}\label{TAdef}
A \emph{triple array} is an $r\times c$ array on $v$ symbols arranged so that no symbol occurs more than once in any row or column, and satisfies the following four conditions:
\begin{itemize}
\item [TA1.] Each symbol occurs $k$ times (equireplicate).
\item [TA2.] Any two distinct rows contain $\lambda_{rr}$ common symbols.
\item [TA3.] Any two distinct columns contain $\lambda_{cc}$ common symbols.
\item [TA4.] Any row and column contain $\lambda_{rc}$ common symbols.
\end{itemize}
\end{df}
For a triple array we use the notation $TA(v,k,\lambda_{rr},\lambda_{cc},\lambda_{rc}:r\times c)$. An array as above that satisfies conditions TA1--TA3 is called a \emph{double array} with notation $DA(v,k,\lambda_{rr},\lambda_{cc}:r\times c)$, and a double array that does not satisfy TA4 is called a \emph{proper} double array.

If we interchange the roles of columns and symbols in a triple array, we get what is called the \emph{RL form} of a triple array. It is an incomplete $r\times v$ array, with no repetition in any row or column, satisfying the following four conditions, corresponding to the above conditions TA1--TA4:
\begin{itemize}
\item [RTA1.] Every column has $k$ occupied cells.
\item [RTA2.] For any pair of rows, there are $\lambda_{rr}$ columns in which both rows are occupied.
\item [RTA3.] Every pair of symbols must occur together in $\lambda_{cc}$ columns.
\item [RTA4.] The columns with occupied cells in row $i$ must contain every symbol exactly $\lambda_{rc}$ times, for all $i$, $1\leq i\leq r$.
\end{itemize}
Correspondingly, an array as above satisfying RTA1--RTA3 is the RL-form of a double array.\\

Triple arrays satisfy the following inequality.

\begin{thm}\cite{MPWY}
Any $TA(v,k,\lambda_{rr},\lambda_{cc},\lambda_{rc}:r\times c)$ satisfies $v\geq r+c-1$. 
\end{thm}

Here we consider the extremal case which is relevant for Agrawal's conjecture and the relation to symmetric designs, and for which there are many examples. In the non-extremal case there is only one example known.

\begin{df}\label{BIBDdef} 
A $(v,k,\lambda)$-\emph{balanced incomplete block design} (BIBD) is a pair $(V,\mathcal{B})$ where $V$ is a $v$-set and $\mathcal{B}$ is a collection of $b$ proper $k$-subsets of $V$ called blocks such that each element of $V$ is contained i exactly $r$ blocks and any 2-subset of $V$ is contained in exactly $\lambda$ blocks. A BIBD where $v=b$ is said to be \emph{symmetric} and is abbreviated as SBIBD.
\end{df}

Note that, in the combinatorial literature, it is common to refer to a BIBD as a \emph{$2$-design}, and a SBIBD as a \emph{symmetric} (or \emph{square}) $2$-design.

\begin{anm}
Another, but equivalent way to define a double array is to say that rows and symbols form a BIBD, and that columns and symbols form a BIBD. Moreover, condition TA4 is often called \emph{adjusted orthogonality}.
\end{anm}

Agrawal conjectured a construction of triple arrays from SBIBDs, which we now state, together with its converse (which is true).

\begin{conj}\cite{Agrawal1}
If there is a $(v+1,r,\lambda_{cc})$-SBIBD with $r-\lambda_{cc}>2$, then there is a $TA(v,k,\lambda_{rr},\lambda_{cc},\lambda_{rc}:r\times c)$ with $v=r+c-1$.
\end{conj}

The condition $r-\lambda_{cc}>2$ excludes the Fano plane and its complement, for which it is known that no triple array exists.

\begin{thm}\label{TASBIBDthm}\cite{MPWY}
If there is a $TA(v,k,\lambda_{rr},\lambda_{cc},\lambda_{rc}:r\times c)$ with $v=r+c-1$, then there is a $(v+1,r,\lambda_{cc})$-SBIBD.
\end{thm}

Agrawal~\cite{Agrawal1} gave a construction method from which many examples have been constructed, but for which one step has not been proven and has to be performed by trial and error. Besides this, only one infinite family called \emph{Paley triple arrays} has been proven to exist. This has been done to different degrees by Seberry~\cite{Seberry}, Street~\cite{Street}, Bagchi~\cite{Bagchi1998} and Preece et al.~\cite{PTA}, and can be summarized as follows.

\begin{thm}
Let $q\geq 5$ be an odd prime power. Then there exists a $q\times(q+1)$ triple array.
\end{thm}

We will use a construction method that starts from Youden squares.

\begin{df}\label{YSdef}
Let $V$ be a set of $v$ symbols. A $k\times v$ \emph{Youden square} is a $k\times v$ array satisfying
\begin{itemize}
\item [YS1.] every symbol of $V$ occurs exactly once in each row, and
\item [YS2.] the symbols in each column constitute a $k$-subset of $V$, and the $v$ such subsets constitute the blocks of a $(v,k,\lambda)$-SBIBD.
\end{itemize}
\end{df}

This construction method, suggested by~\cite{RagNag}, starts from a Youden square and goes directly to the RL-form of a triple array, thereby avoiding the unproven step of Agrawal's method. However, this places strong demands on the structure of the Youden square, and it was shown in~\cite{WYAgrawal} that this method is not even guaranteed to give a double array. But in~\cite{TAYS} the usability of this method was further studied, also with some positive results. That a Youden square can always be constructed from a SBIBD was proven by~\cite{SmithHartley}. But the most straightforward way to construct a Youden square is to develop it from a difference set.\\

In this paper we investigate if there are difference sets which, if developed to Youden squares, give triple arrays by the suggested construction. The rest of the paper is structured as follows: In Sect.~\ref{prelsect} we give preliminary definitions and results about difference sets. In Sect.~\ref{constrsect} we present the construction and prove that having $-1$ as a multiplier is both a necessary and sufficient condition for an abelian difference set to give a triple array. In Sect.~\ref{existencesect} we present known results about existence for such difference sets and give parameters for a new infinite family of triple arrays. In Sect.~\ref{directsect} we present a direct version of the construction and formulate a more general condition for a difference set to give a triple array, suitable for studies of non-abelian sets. Finally, in Sect.~\ref{nonabsect} we analyze the non-abelian case. We show by a computer search on small groups that having $-1$ as weak multiplier, or having a reversible translate, are not sufficient conditions for such a set to give a triple array. 
\end{section}
\begin{section}{Preliminaries}\label{prelsect}
We recall some facts from the theory of difference sets.

\begin{df}\label{DSdef} 
Let $G$ be a finite multiplicative group of order $v$, and $D$ a $k$-subset of  $G$. Then $D$ is called a $(v,k,\lambda)$-\emph{difference set} if any non-identity element of $G$ can be written in exactly $\lambda$ ways as $xy^{-1}$ where $x$ and $y$ are in $D$. The \emph{order} of a difference set is the integer $n=k-\lambda$ and $D$ is called non-trivial if $n>1$. We say that $D$ is \emph{cyclic} or \emph{abelian} if $G$ is.
\end{df}
Unless stated otherwise, the term ``difference set'' will here always refer to a non-trivial difference set. Given a difference set, we can find many other related difference sets.
\begin{thm}[cf.~\cite{MoorePoll}]
Let $D$ be a $(v,k,\lambda)$-difference set in $G$. Then its complement $\overline{D}=G\setminus D$ is a difference set in $G$. 
\end{thm}

\begin{thm}[cf.~\cite{MoorePoll}]
Let $D\subset G$ be a $(v,k,\lambda)$-difference set.
\begin{enumerate}
	\item For $g\in G$, both $gD$ and $Dg$ are $(v,k,\lambda)$-difference sets.
	\item Let $\alpha$ be an automorphism of $G$. Then $\alpha(D)$ is a $(v,k,\lambda)$-difference set.
\end{enumerate}
\end{thm}

For parameters we have the following fundamental identity.
\begin{thm}[cf.~\cite{MoorePoll}]\label{identitetDS}
Let $D$ be a $(v,k,\lambda)$-difference set in $G$. Then $k(k-1)=\lambda(v-1)$. 
\end{thm}

We will need the notion of reversible difference sets.
\begin{df}
Let $S$ be a subset of a group $G$. Then $S^{(-1)}$ is the set of inverses of the elements of $S$, that is $S^{(-1)}=\{x^{-1}:x\in S\}$.
\end{df}

\begin{df}\label{revDdef}
A difference set $D$ is called \emph{reversible} if $D=D^{(-1)}$. 
\end{df}

Difference sets can be used to construct SBIBDs.
\begin{df}\label{devDdef}
Let $X$ be a subset of a finite group $G$. For any $g\in G$, define $Xg=\{xg:x\in G\}$. We call any set $Xg$ a \emph{right translate} of $X$ and define the \emph{development of} $X$, denoted $\dev(X)$, to be the collection of all right translates of $X$.
\end{df}

\begin{thm}[cf.~\cite{MoorePoll}]\label{devDSBIBDthm}
Let $D\subset G$ be a $(v,k,\lambda)$-difference set. Then $(G,\dev(D))$ is a $(v,k,\lambda)$-SBIBD. 
\end{thm}

Studying difference sets often means dealing with multisets, and then the following algebraic structure can be useful.
\begin{df}\label{ZGdef}
Let $G$ be a finite multiplicative group. The \emph{integral group ring} $\mathbb{Z}G$ consists of formal sums $\sum_{g\in G}a_gg$ where $a_g\in\mathbb{Z}$. Addition and multiplication are defined as follows:
$$\sum_{g\in G}a_gg+\sum_{g\in G}b_gg=\sum_{g\in G}(a_g+b_g)g$$
$$\left(\sum_{f\in G}a_ff\right)\left(\sum_{g\in G}b_gg\right)=\sum_{h\in G}\left(\sum_{fg=h\in G}a_fb_g\right)h.$$
\end{df}
Let $S\subseteq G$, then it is a standard convention to abuse notation and write also the corresponding group ring element as $S$, that is $S:=\sum_{g\in S}g$. In $\mathbb{Z}G$, if $A=\sum_{g\in G}a_gg$ and $t$ any integer, then we define $A^{(t)}:=\sum_{g\in G}a_gg^t$. In particular, $A^{(-1)}=\sum_{g\in G}a_gg^{-1}$. A useful notation concerning cardinality is $|A|=\sum_{g\in G}a_g$. We write the identity in $G$ as $1_G$ to distinguish it from the integer 1. Further, $\mathbb{Z}G$ is a ring with identity $1_G$ and is commutative if and only if $G$ is abelian.

\begin{thm}[cf.~\cite{MoorePoll}]\label{DDinvthm}
Let $D$ be a non-empty proper subset of a group $G$ with $|D|=k$ and $|G|=v$. Then $D$ is a $(v,k,\lambda)$-difference set of order $n$ if and only if
$$DD^{(-1)}=n1_G+\lambda G$$
holds in $\mathbb{Z}G$.
\end{thm}

\begin{df}\label{multdef}
Let $D$ be a difference set in a group $G$. An automorphism $\phi$ of $G$ is called a \emph{multiplier} of $D$ if $\phi(D)=aDb$ for some $a,b\in G$. If $a=1$, then $\phi$ is called a \emph{right multiplier}. If $G$ is abelian and $\phi$ is on the form $\phi_t:x\mapsto x^t$ for some integer $t$, then $\phi$ is called a \emph{numerical multiplier}. It is common practice to  abuse terminology and call $t$ itself a \emph{numerical multiplier}. 
\end{df}

For $\phi_t:x\mapsto x^t$ we write $\phi_t(D)=D^{(t)}=\{x^t:x\in D\}$. We will take particular interest in the inverse mapping $x\mapsto x^{-1}$. It is an automorphism of every abelian group but quite rare as multiplier. In non-abelian groups this mapping is not a homomorphism so it cannot be a multiplier, but is called a \emph{weak multiplier} if $D^{(-1)}$ is a right translate of $D$.

\begin{thm}[cf.~\cite{BJL}]\label{fixttranslate}
Let $D$ be a $(v,k,\lambda)$-difference set in the abelian group $G$. Then there exists a translate of $D$ which is fixed by every numerical multiplier of $D$.
\end{thm}

Hence, $-1$ is a multiplier of $D$ if and only if $D$ has a reversible translate. Also, note that if $-1$ is a multiplier of $D$, then it also a multiplier of every translate of $D$. To see this, suppose $(Dg)^{(-1)}=Dg$. Then $ g^{-1}D^{(-1)}=Dg$, so $D^{(-1)}=Dg^2$.\\

Finally, we will need the following two results for calculating parameters in triple arrays.

\begin{lma}\cite{McS}\label{lambda_rclma}
Suppose $\mathcal{A}$ is a $TA(v,k,\lambda_{rr}.\lambda_{cc},\lambda_{rc}:r\times c)$ with $v=r+c-1$. Then $\lambda_{rc}=r-\lambda_{cc}$.
\end{lma}

\begin{kor}\cite{McS}\label{TAparakor}
When $v=r+c-1$ every triple array is a $TA(v,k,c-k,r-k,k:r\times c)$.
\end{kor}
\end{section}
\begin{section}{Triple arrays from abelian difference sets}\label{constrsect}

In this section we construct triple arrays from Youden squares developed from abelian difference sets admitting $-1$ as a multiplier, and prove that this is both a necessary and sufficient condition for the construction to work. By Theorem~\ref{devDSBIBDthm} it is straightforward to construct a Youden square  from a difference set.

\begin{constr}\label{Yconstr}
Let $D$ be a $(v,k,\lambda)$-difference set in a finite multiplicative group $G$. Let $Y$ be a $k\times v$ array where the columns are inxeded by the elements $j\in G$. Write the elements of $D$ in an arbitrary order in column $1$, and let each row be indexed by the element $i\in D$ in column 1. Then, take $Y(i,j)=ij$, where $Y(i,j)$ denotes the element in position $(i,j)$ in $Y$.   
\end{constr}

From a Youden square we construct the RL-form of an array as suggested in~\cite{RagNag}.

\begin{constr}\label{Rconstr}
Let $Y$ be a Youden square, and $C$ a column in $Y$. Let $Y^{[C]}$ be produced from $Y$ by deleting $C$ from $Y$ together with all the symbols in $C$.
\end{constr}

As mentioned before, applying Construction~\ref{Rconstr} to any Youden square $Y$ does not even guarantee that we get the RL-form of a double array. But here $Y$ is developed from a difference set. 
\begin{thm}\label{DAthm} 
Let $D$ be a $(v,k,\lambda)$-difference set in a group $G$. Then there is a $k\times(v-k)$ double array.
\end{thm}
\begin{proof}
This theorem consists of two results already known put together for constructing the RL-form of a double array. In~\cite{TAYS}, RTA1 and RTA3 was proven in Proposition 2 for any Youden square $Y$, and in Proposition 5 RTA2 was proven in the special case when $Y$ is constructed from a difference set as in Construction~\ref{Yconstr}, without any assumtion of being abelian and to hold regardless of the choice of column $C$ for $Y^{[C]}$ in Construction~\ref{Rconstr}. Here we write $G$ multiplicatively, whereas it was written additively in~\cite{TAYS}. Hence, from $D$ we can construct an RL-form of a double array and thereby a double array with $|D|=k$ rows and $|G|-|D|=v-k$ columns.
\end{proof}

\begin{thm}\label{TAthm}
Let $D$ be a $(v,k,\lambda)$-difference set in an abelian group $G$ which admits $-1$ as a multiplier. Then there is a $k\times(v-k)$ triple array.
\end{thm}
\begin{proof}
Let $Y$ be a Youden square obtained from Construction~\ref{Yconstr} when applied to a difference set such as above. We will prove that applying Construction~\ref{Rconstr} to any column $C$ of $Y$ gives a $Y^{[C]}$ which is the RL-form of a triple array. We already have RTA1, RTA2 and RTA3 by Theorem~\ref{DAthm} and what remains is to prove RTA4. 

That a difference set with multiplier $-1$ has a reversible translate is given by Lemma~\ref{fixttranslate}. So, suppose $Y$ is constructed from a difference set $Dg$ with reversible translate $D$. Take an element $x\in Dg$. We will prove that the columns in $Y$ which have elements from $G\setminus Dg$ in row $x$ together contain exactly $n=k-\lambda$ copies of each element in $G\setminus Dg$.   

We partition $G$ into $L_x$ and $R_x$ such that $L_x$ is the set of column indices for columns having an element of $Dg$ in row $x$, and $R_x$ consists of the remaining column indices. If wanted, the columns of $Y$ could be permuted as in Figure~\ref{Dgshade}. 

\begin{figure}[h]
\begin{tikzpicture}
\draw[fill=gray!20]
(0mm,30mm) rectangle (5mm,0mm)
(5mm,10mm) rectangle (25mm,15mm);

\draw
(0mm,10mm) -- (5mm,10mm)
(0mm,15mm) -- (5mm,15mm)
(5mm,0cm) -- (25mm,0mm)
(25mm,15mm) -- (60mm,15mm)
(25mm,10mm) -- (60mm,10mm)
(25mm,30mm) -- (25mm,0mm) -- (60mm,0mm) --(60mm,30mm)
(5mm,30mm) -- node[midway,above]{$L_x$} (25mm,30mm)
-- node[midway,above]{$R_x$} (60mm,30mm)
(2.5mm,12.5mm) node {$x$};

\end{tikzpicture}
	\caption{A Youden square where the shaded part of row $x$ and the first column both have support $Dg$.}
	\label{Dgshade}
\end{figure}

Note that $Y$ can be regarded as part of a multiplication table for $G$ and that in $\mathbb{Z}G$ the product $DgR_x$ represents a multiset consisting of the elements in the columns with indices in $R_x$, counted with multiplicity. First we note that $x=dg$ for some $d\in D$ and that $xL_x=Dg$, so

\begin{equation*}
L_x=x^{-1}Dg=(dg)^{-1}Dg=g^{-1}d^{-1}Dg=Dd^{-1}.
\end{equation*}

We will now determine the multiset $DgR_x$.
$$
DgR_x=Dg(G-L_x)=DgG-DgL_x=DgG-DgDd^{-1}=DgG-DDgd^{-1}
$$
We note that $DgG=|Dg|G=kG$ and since $D$ is reversible we can make use of Theorem~\ref{DDinvthm} to write 
\begin{multline*}
DgG-DDgd^{-1}=kG-(n1_G+\lambda G)gd^{-1}=kG-ngd^{-1}-\lambda Ggd^{-1}=\\
=kG-ngd^{-1}-\lambda G=(k-\lambda)G-ngd^{-1}=nG-ngd^{-1}
\end{multline*} 
where $gd^{-1}\in Dg$.

Hence, we have proved that every element in the columns indexed by $R_x$ contain each element of $G\setminus Dg$ exactly $n=k-\lambda$ times. Since $x$ is an arbitrary element of $Dg$ this holds for any $x\in Dg$ thereby proving RTA4 for $Y^{[C]}$ where $Dg$ is the support of $C$. And, because $g\in G$ was arbitrary, the proof holds with respect to any column $C$ of $Y$. 

\end{proof}
And now we will prove the converce, that if a Youden square constructed from an abelian difference set $D$ gives a triple array, then $D$ admits $-1$ as a multiplier.

For convenience we let $Y^{[C]}_x$ denote the array consisting of all columns of $Y^{[C]}$ where row $x\in\supp(C)$ has elements in $G\setminus\supp(C)$. 

\begin{lma}\label{lambdalma}
Let $D$ be a non-reversible abelian difference set that via Construction~\ref{Yconstr} gives a Youden square from which  Construction~\ref{Rconstr} gives the RL form of a triple array, then $|D\cap D^{(-1)}|=\lambda$.
\end{lma}
\begin{proof}
Suppose $D$ gives the RL-form of a triple array when applying Construction~\ref{Yconstr} and Construction~\ref{Rconstr}, but that $D$ is not reversible. Then there is an $x\in D$ such that $x^{-1}\in\overline D$ and we claim that $x^{-1}$ occurs $|\overline D\cap D^{(-1)}|$ times in $Y^{[C]}_x$.

Let $y\in D$ such that $y^{-1}\in\overline D$, $x$ and $y$ not necessarily distinct. Any column $z$ where $xz=y^{-1}$ is active in $Y^{[C]}_x$, but then we also have $yz=x^{-1}$. Hence, $x^{-1}$ occurs at least  $|\overline D\cap D^{(-1)}|$ times in $Y^{[C]}_x$. To see that $x^{-1}$ does not occur more times in $Y^{[C]}_x$ we assume that $y^{-1}\in D$ and $yw=x^{-1}$. But then $xw=y^{-1}$ so column $w$ is not active in $Y^{[C]}_x$ which proves the claim.

Now, since RTA4 tells us that every element of $\overline{D}$ occur $\lambda_{rc}$ times in $Y^{[C]}_x$ and Theorem~\ref{TASBIBDthm} together with Lemma~\ref{lambda_rclma} give $\lambda_{rc}=k-\lambda$ we have
$$\lambda_{rc}=k-\lambda=|\overline D\cap D^{(-1)}|=|D|-|D\cap D^{(-1)}|.$$
Hence, $|D\cap D^{(-1)}|=\lambda$. 

\end{proof}

\begin{lma}\label{snittminus1lma}
Let $D$ be a difference set in an abelian group $G$. Then  
$$|Dg\cap(Dg)^{(-1)}|\in\{|D|,\lambda\}\quad\text{ for all }g\in G$$
if and only if $D$ admits $-1$ as a multiplier.
\end{lma}
\begin{proof}
Suppose $D$ admits $-1$ as a multiplier. Then $-1$ is also a multipler of every translate, so for each $g\in G$ we have $|Dg\cap(Dg)^{(-1)}|=\lambda$ or $|Dg\cap(Dg)^{(-1)}|=|D|$ since we know by Lemma~\ref{fixttranslate} that there is a least one translate $Dg$ fixed by the multiplier $-1$. 

Now, suppose $|Dg\cap(Dg)^{(-1)}|\in\{|D|,\lambda\}$ for all $g\in G$ and that $-1$ is not a multiplier of $D$. So there is no $g\in G$ such that $|Dg\cap(Dg)^{(-1)}|=|D|$ which means that $|Dg\cap(Dg)^{(-1)}|=\lambda$ for all $g\in G$. Hence, $d_ig=g^{-1}d_j^{-1}$ or equivalentely $d_id_j=g^{-2}$ has exactly $\lambda$ solutions of pairs $d_i,d_j\in D$ for each $g\in G$.   

This would be $\lambda|G|$ elements $g^{-2}$ in total when counted with multiplicity. We compare this with all elements $d_id_j$ counted with multiplicity which is given by $|DD|$ in $\mathbb{Z}G$. For  $(v,k,\lambda)$-difference sets we have the identity $\lambda(v-1)=k(k-1)$ given in Theorem~\ref{identitetDS} and can write
$$|DD|=k^2>|\lambda G|=\lambda v=k^2-(k-\lambda)=k^2-n.$$
So, there is a $g^{-2}\in G$ for which there are more than $\lambda$ solutions of $d_id_j=g^{-2}$. And since $|D|$ is the only avialable number besides $\lambda$ we know that there is a reversible translate which means that $-1$ is a multiplier. 
\end{proof}

\begin{thm}
Let $D$ be an abelian difference set that via Construction~\ref{Yconstr} gives a Youden square from which Construction~\ref{Rconstr} gives the RL form of a triple array, then $D$ admits $-1$ as a multiplier.
\end{thm}
\begin{proof}
By Lemma~\ref{lambdalma} we know that a non-reversible abelian difference set $D$ which gives the RL-form of a triple array through Construction~\ref{Yconstr} and Construction~\ref{Rconstr} must satisfy $|D\cap D^{(-1)}|=\lambda$. But  Lemma~\ref{snittminus1lma} tells us that it is not possible that $|(Dg)\cap (Dg)^{(-1)}|=\lambda$ for all $g\in G$. Hence, there exists some $g\in G$ such that $Dg$ is reversible wich proves that $-1$ is a multiplier of $D$. 
\end{proof}

\end{section}
\begin{section}{Existence}\label{existencesect}
Here we consider abelian difference sets having multiplier $-1$, or we can say being reversible because of Lemma~\ref{fixttranslate}. An overview of the theory can be found in~\cite{BJL}, here we only cover enough for answering questions relevant for triple arrays. We give all small such difference sets as examples and prove the existence of an infinite family of triple arrays.

\begin{thm}[cf.~\cite{BJL}]\label{veventhm}
Let $D$ be a reversible $(v,k,\lambda)$-difference set in an abelian group, then $v$ and $\lambda$ are even and $n$ is a square. Moreover, $v$ and $n$ have the same odd prime divisors.
\end{thm}

We note that triple arrays constructed here never coincide with Paley triple arrays. To see this, we compare the sum of the number of rows and columns. For a Paley triple array the sum is $2q+1$, so it is odd. For triple arrays constructed from $(v,k,\lambda)$-difference sets with multiplier $-1$ the sum is $k+(v-k)=v$, which is even by Theorem~\ref{veventhm}. The following theorem excludes many well-known difference sets.

\begin{thm}[cf.~\cite{Baumert}]\label{nixcyclicthm}
Minus one is never a multiplier of a non-trivial cyclic difference set.
\end{thm}

Thus we turn our attention to non-cyclic groups. Kibler~\cite{Kibler} gave a list of the small non-cyclic $(v,k,\lambda)$-difference sets for $k<20$ up to equivalence. This list consists of 74 difference sets in 37 groups. But the only groups here satisfying the necessary conditions of being both abelian and of even order are of order 16 or 36.

\paragraph{\bf{The case} $(v,k,\lambda)=(16,6,2)$} Here there are 12 groups having a total of 27 difference sets. Four of these groups are abelian, having a total of eight difference sets. But only four difference sets admit $-1$ as multiplier. These groups and sets were pointed out in~\cite{Johnsen} and are given here below with numbering as in~\cite{Kibler}.

\begin{enumerate}
	\item [(B)] Abelian. $\langle a,b:a^4=b^4=1\rangle$
		\begin{enumerate}
			\item [3.] $\{1,a,a^2,b,ab^2,a^2b^3\}$
			\item [5.] $\{1,a,b,a^2b,ab^2,a^2b^2\}$
		\end{enumerate}
	\item [(C)] Abelian. $\langle a,b,c:a^4=b^2=c^2=1\rangle$
		\begin{enumerate}
			\item [7.] $\{1,a,a^2,ab,ac,a^3bc\}$
		\end{enumerate}
	\item [(D)] Abelian. $\langle a,b,c,d:a^2=b^2=c^2=d^2=1\rangle$
		\begin{enumerate}
			\item [8.] $\{1,a,b,c,d,abcd\}$
		\end{enumerate}
\end{enumerate}

\paragraph{\bf{The case} $(v,k,\lambda)=(36,15,6)$} Here there are nine groups having a total of 34 difference sets. The two abelian groups contain seven difference sets, but only the following admits $-1$ as multiplier. 
\begin{enumerate}
	\item [(A II)] Abelian. $\langle a,b,c,d:a^3=b^3=c^2=d^2=1\rangle$
		\begin{enumerate}
			\item [17.] $\{1,a,a^2,c,a^2c,bc,a^2bc,b^2c,a^2b^2c,ad,a^2bd,b^2d,acd,bcd,a^2b^2cd\}$
		\end{enumerate}
\end{enumerate}
Applying Constructions~\ref{Yconstr} and~\ref{Rconstr} to the difference sets above give $6\times 10$ and $15\times 21$ triple arrays by Theorem~\ref{TAthm}. We note that these difference sets are all members of the following well-known and important family.

\begin{df}
A $(4u^2,2u^2\pm u, u^2\pm u)$-difference sets is called a \emph{Hadamard difference set}.
\end{df}

To our knowledge, almost all reversible difference sets are Hadamard. This question was systematically studied by~\cite{McFMa1990} and~\cite{Ma1991}, and the only exception known is due to McFarland~\cite{McFarland73} who proved there is an abelian reversible $(4000,775,150)$-difference set and proposed the following conjecture.
\begin{conj}[McFarland's Conjecture]
If $D$ is a reversible difference set, then either $D$ is a $(4000,775,150)$-difference set or $D$ is a Hadamard difference set.
\end{conj}

For Hadamard difference sets there is a strong non-existence theorem.

\begin{thm}\cite{McFarland1990b}\label{bara2och3}
If there exists a Hadamard difference set with multiplier minus one in an abelian group of order $4u^2$, then the square-free part of $u$ can only have the prime factors 2 and 3.
\end{thm}

There are several constructions of reversible difference sets. Here follows two almost trivial composition theorems, both due to Menon~\cite{Menon62}. But as they both preserve reversibility they guarantee that we can construct an infinite number of triple arrays. 

\begin{thm}\cite{Menon62}\label{Menonproduct}

If $G_1$ and $G_2$ each contain a Hadamard difference set, then so does $G_1\times G_2$. 
\end{thm}

\begin{prop}\cite{Menon62}\label{Menoncomplproduct}
The existence of Hadamard difference sets for $u=u_1$ and $u=u_2$ implies that of a Hadamard difference set for $u=2u_1u_2$.
\end{prop}
The proof is built on the following construction. Let $D_1$ and $D_2$ be the given difference sets in the groups $G_1$ and $G_2$ respectively. Then the desired difference set is obtained as $D=(D_1\times(G_2\setminus D_2))\cup((G_1\setminus D_1)\times D_2)\quad\text{of }G:=G_1\times G_2$.\\

All know constructions were put together in~\cite{JungSchmidt1997} to give the following result.
\begin{thm}[cf.~\cite{BJL}]\label{allknownrevHada}
Let $G$ be a group of the form $G=\mathbb{Z}^b_4\times\mathbb{Z}^2_{2^{c_1}}\times\ldots\times\mathbb{Z}^2_{2^{c_r}}$, and let $H$ be a group of the form $H=\mathbb{Z}^2_2\times\mathbb{Z}^{2a}_3\times\mathbb{Z}^4_{p_1}\times\ldots\times\mathbb{Z}^4_{p_s}$, where the $p_j$ are (not necessarily distinct) odd primes and where $a,b,c_1,\ldots,c_r$ are non-negative integers. Then both $G$ and $G\times H$ contain reversible Hadamard difference sets. 
\end{thm}

In~\cite{BJL} Theorems~\ref{bara2och3} and~\ref{allknownrevHada} are combined to give the following. 
\begin{kor}[cf.~\cite{BJL}]\label{div6kor}
Let $u$ be a positive integer. There exists a reversible Hadamard difference set in some abelian group of order $4u^2$ if and only if the square-free part of $u$ divides $6$.
\end{kor}

We conclude this section by giving an existence result for triple arrays based on the knowledge of abelian reversible difference sets.
\begin{thm}\label{existencethm}
Let $u$ be a positive integer such that the square-free part of $u$ divides $6$, then there is a $TA(4u^2-1,u^2,u^2+u,u^2-u,u^2:(2u^2-u)\times(2u^2+u))$. Moreover, there is a $TA(3999,625,2600,150,625:775\times 3225)$.
\end{thm}
\begin{proof}
By Corollary~\ref{div6kor} we know that there is an abelian reversible $(4u^2,2u^2-u,u^2-u)$-difference sets when the square-free part of $u$ divides $6$, and by~\cite{McFarland73} that there is an abelian reversible $(4000,775,150)$-difference set. Applying Theorem~\ref{TAthm} to these two cases gives the existence of the corresponding triple arrays, where the parameters are calculated using Theorem~\ref{TASBIBDthm} and Corollary~\ref{TAparakor}.
\end{proof}

\end{section}

\begin{section}{A direct construction}\label{directsect}
To help analyse non-abelian difference sets, we give a more direct construction
of row-column designs from difference sets. This is nothing new: we have simply
re-written the previous construction leaving out the intermediate step via
Youden squares.

\begin{df} Let $D$ be a non-empty proper subset of $G$. Construct
an array $A(G,D)$ with rows indexed by $D$ and columns by $G\setminus D$,
with the $(x,y)$ entry equal to $x^{-1}y$. Note that the entries of the
array are non-identity elements of $G$, and that no element of $G$ is repeated
in a row or column.
\end{df}

\begin{prop}
Suppose that $D$ is a left difference set in a group $G$ of order $v$. Then
$A(G,D)$ is a double array of size $k\times(v-k)$ on $v-1$ symbols.
It is a triple array if and only if the size of
$x^{-1}D\cap D^{(-1)}y$ is constant for $x\in D$, $y\in G\setminus D$.\qed
\end{prop}

\begin{proof}
Suppose that $D$ is a $(v,k,\lambda)$ left difference set, so that any two
left translates of $D$ meet in $\lambda$ elements.
The cancellation law for groups shows that no element is repeated in a row or
column; and clearly the identity does not occur anywhere in the array.
Given $z\ne0$, there are $k$ representations as $z=x^{-1}y$ with
$x\in D$ (one for each element of $D$; and $\lambda$ of these have $y\in D$.
So the number of occurrences of $z$ in the table is $k-\lambda$. Further,
we see that
\begin{itemize}\itemsep0pt
\item row $x$ of $A$ contains the elements of $x^{-1}(G\setminus D)$;
\item column $y$ of $A$ contains the elements of $D^{(-1)}y$.
\end{itemize}

If $D$ is a left difference set, then so is $G\setminus D$. Moreover,
$D^{(-1)}y=(y^{-1}D)^{-1}$, so any two translates of $D^{(-1)}$ meet in a
constant number of points. It follows immediately that $A$ is a double array.
The condition for $A$ to be a triple array is clear.
\end{proof}

Suppose that $D^{(-1)}$ is equal to a left translate of $D$, and
that every right translate of $D$ is also a left translate. Then any row and
any column contain a constant number of common elements, so we have a triple
array. Thus, if $D$ is a difference set with multiplier $-1$ in an abelian
group, then $A$ is a triple array.
\end{section}
\begin{section}{Non-abelian difference sets}\label{nonabsect}
\begin{quest}
Suppose that $D$ is a left difference set in a non-abelian
group $G$. Is it ever possible for $A(G,D)$ to be a triple array?
\end{quest}

We next discuss some computations related to this question. (We have not found
any examples.)

\begin{prop}
Let $D$ be a left difference set in a group $G$. Then
\begin{enumerate}
\item [(a)] $D$ generates $G$, so $D$ cannot be contained in a proper subgroup of $G$.
\item [(b)]The set $S=\{x\in G: xD=Dy\hbox{ for some }y\in G\}$ is a subgroup of $G$.
\item [(c)] If $D\subseteq S$, then $S=G$.
\end{enumerate}
\end{prop}

\begin{proof}
(a) holds because every element of $G$ has the form $x^{-1}y$ for some
$x,y\in D$.

To see that (b) holds, note that $G\times G$ acts on the power set of $G$ by
the \emph{diagonal action} $(x,y):D\mapsto x^{-1}Dy$; the stabiliser of a set
$D$ is thus a subgroup of $G\times G$, and the set $S$ of the Proposition is
the projection of this subgroup onto the first factor.
\end{proof}

Before we discuss the non-abelian difference sets in Kibler's list~\cite{Kibler}, we recall some preliminary observations.
\begin{enumerate}\itemsep0pt
\item [(a)] $D$ is a left difference set if and only if it is a right difference
set (that is, the non-identity elements of $G$ occur equally often in the form
$yx^{-1}$, for $x,y\in D$~\cite{Bruck}.
\item [(b)] We say that $-1$ is a \emph{weak multiplier} of $D$ if the set
$D^{(-1)}=\{d^{-1}:d\in D\}$ (which is also a difference set) is a translate
$Dx$ of $D$, for some $x\in G$.
\item [(c)] We say that $D$ is \emph{reversible} if $D^{(-1)}=D$. Note that, if 
$D^{(-1)}=Dx^2$ where $x$ is central in $G$, then
\[(Dx)^{(-1)}=(xD)^{(-1)}=D^{(-1)}x^{-1}=Dx,\]
so $Dx$ is a reversible difference set.
\end{enumerate}

\paragraph{\bf{The case} $(v,k,\lambda)=(16,6,2)$} For this case Kibler found, up to equivalence, that eight of the twelve groups having difference sets are non-abelian, containing $19$ difference sets. 

Note that Kibler's group (J) is incorrect. The presentation of the group should
be $\langle a,b:a^8=b^2=1,ba=a^5b\rangle$, and the two difference sets are
$\{1,a,a^2,a^5,a^4b,a^2b\}$ and $\{1,a,a^3,a^4,a^3b,a^5b\}$. This is noted
by an asterisk in the table below. The other notation is as in Kibler's paper.

Using \textsf{GAP} \cite{gap}, we examined all cases. Data on the
number of reversible translates of $D$, the number of translates with a
weak multiplier $-1$ (this is weaker than having a reversible translate),
and on the number $s$ of $x$ for which $xD=Dy$ for some $y$, is given in
Table~\ref{t:16}. Each of these numbers is independent of the particular
choice of difference set among the collection of all translates.

\begin{table}[htbp]
\begin{center}
\begin{tabular}{|c|c|c|c|c|}
\hline
Group & Difference set & reversible & weak mult $-1$ & left $=$ right \\
\hline
E & 9 & $8$ & $8$ & $8$ \\
  & 10 & $8$ & $8$ & $8$ \\
\hline
F & 11 & -- & -- & $4$ \\
  & 12 & -- & $8$ & $8$ \\
\hline
G & 13 & -- & -- & $8$ \\
  & 14 & -- & -- & $4$ \\
\hline
H & 15 & $4$ & $8$ & $8$ \\
  & 16 & $4$ & $8$ & $8$ \\
  & 17 & $4$ & $8$ & $8$ \\
  & 18 & $4$ & $8$ & $8$ \\
\hline
I & 19 & -- & $8$ & $8$ \\
  & 20 & $4$ & $8$ & $8$ \\
  & 21 & -- & $8$ & $8$ \\
\hline
J* & 22 & 4 & $8$ & $8$ \\
   & 23 & 4 & $8$ & $8$ \\
\hline
K & 24 & $4$ & $4$ & $4$ \\
  & 25 & $4$ & $4$ & $4$ \\
\hline
L & 26 & -- & $4$ & $4$ \\
  & 27 & -- & $4$ & $4$ \\
\hline
\end{tabular}
\end{center}
\caption{\label{t:16}$(16,6,2)$ non-abelian difference sets}
\end{table}

Note that $s=v$ would be equivalent to the assertion that
every left translate is a right translate (giving a triple array); but this
never occurs.

We mention here a curious combinatorial structure which occurs in the examples
with $s=8$. In each of these cases, let $\mathcal{B}_1$ be the collection of
left translates of $D$ which are also right translates; $\mathcal{B}_2$ the
collection of left translates which are not right translates; $\mathcal{B}_4$
the collection of right translates which are not left translates; and
$\mathcal{B}_3$ the collection of sets which can be obtained from $D$ by a left
and a right translation, which have not already been included. Then each of
$\mathcal{B}_1$, \dots, $\mathcal{B}_4$ consists of eight $6$-element subsets
of $G$, having the properties that each of $\mathcal{B}_1\cup\mathcal{B}_2$,
$\mathcal{B}_2\cup\mathcal{B}_3$, $\mathcal{B}_3\cup\mathcal{B}_4$ and
$\mathcal{B}_4\cup\mathcal{B}_1$ is the set of blocks of a
$(16,6,2)$ SBIBD: a phenomenon which might repay further investigation!

Table~\ref{t:36} has the same information for the non-abelian $(36,15,6)$
difference sets. Again, we fail to find a triple array; but for difference
set number 15, half of the left translates are right translates, so we obtain
another example of the $4$-cycle of sets of blocks as described above.

\begin{table}[htbp]
\begin{center}
\begin{tabular}{|c|c|c|c|c|}
\hline
Group & Difference set & reversible & weak mult $-1$ & left $=$ right \\
\hline
IB & 5 & $2$ & $6$ & $6$ \\
   & 6 & -- & -- & $6$ \\
   & 7 & -- & -- & $6$ \\
   & 8 & $2$ & $6$ & $6$ \\
   & 9 & -- & -- & $6$ \\
   & 10	& -- & -- & $6$ \\
\hline
IC & 11 & -- & -- & $4$ \\
\hline
ID & 12	& $3$ & $9$ & $9$ \\
   & 13	& $3$ &	$9$ & $9$ \\
   & 14	& -- & -- & $9$ \\
   & 15	& $6$ &	$18$ & $18$ \\
   & 16 & $2$ &	$2$ & $2$ \\
\hline
IIB & 20 & $4$ & $4$ & $4$ \\
\hline
IIC & 21 & -- & -- & $6$ \\
    & 22 & $2$ & $6$ & $6$ \\
    & 23 & $2$ & $6$ & $6$ \\
    & 24 & -- & -- & $6$ \\
    & 25 & -- & -- & $6$ \\
    & 26 & $2$ & $6$ & $6$ \\
\hline
IID & 27 & $3$ & $3$ & $3$ \\
    & 28 & $4$ & $6$ & $6$ \\
    & 29 & $3$ & $3$ & $3$ \\
    & 30 & $4$ & $6$ & $6$ \\
    & 31 & $4$ & $6$ & $6$ \\
    & 32 & $3$ & $3$ & $3$ \\
\hline
IIE & 33 & -- & -- & $3$ \\
    & 34 & $1$ & $3$ & $3$ \\
\hline
\end{tabular}
\end{center}
\caption{\label{t:36}$(36,15,6)$ non-abelian difference sets}
\end{table}

\end{section}


\end{document}